\documentclass[11pt]{article}
\usepackage[english]{babel}
\usepackage[utf8]{inputenc}
\usepackage{amsmath}
\usepackage{amsthm,amsfonts,amssymb,mathrsfs,dsfont}
\usepackage{minitoc}
\usepackage{lmodern}

\usepackage[mono=false]{libertine}
\useosf

\usepackage{wasysym}

\usepackage[dvipsnames]{xcolor}
\usepackage[a4paper,vmargin={2.6cm,2.7cm},hmargin={2.6cm,2.6cm}]{geometry}
\usepackage[font=sf, labelfont={sf,bf}, margin=1cm]{caption}

\usepackage{lscape}
\usepackage{pdflscape}
\usepackage{graphicx}

\usepackage[font=footnotesize,labelfont=bf]{caption}
\usepackage{subfig} 
\usepackage{epsfig}
\usepackage{latexsym}
\usepackage{stmaryrd}

\usepackage[english]{babel}

 \usepackage[colorlinks=true]{hyperref}
\usepackage{pstricks}
\usepackage{enumerate}
\usepackage[normalem]{ulem}

\theoremstyle{plain}

\usepackage[font=footnotesize,labelfont=bf]{caption}

\theoremstyle{plain}
\newtheorem{theorem}{Theorem}[section]
\newtheorem{corollary}[theorem]{Corollary}

\newtheorem{lemma}[theorem]{Lemma}
\newtheorem{proposition}[theorem]{Proposition}

\theoremstyle{definition}
\newtheorem{definition}[theorem]{Definition}

\newtheorem{remark}[theorem]{Remark}

\makeatletter
\newcommand*\bigcdot{\mathpalette\bigcdot@{.5}}
\newcommand*\bigcdot@[2]{\mathbin{\vcenter{\hbox{\scalebox{#2}{$\m@th#1\bullet$}}}}}
\makeatother

\renewcommand{\leq}{\leqslant}
\renewcommand{\geq}{\geqslant}

\newcommand{\N}{\mathbb{N}}
\newcommand{\Z}{\mathbb{Z}}

\renewcommand{\S}{\mathcal{S}}

\newcommand{\ind}[1]{\mathbf{1}_{\left\{#1\right\}}}

\newcommand{\floor}[1]{{\left\lfloor #1 \right\rfloor}}

\DeclareMathOperator{\E}{\mathbf{E}}
\renewcommand{\P}{\mathbf{P}}

\renewcommand{\bar}[1]{\overline{#1}}
\renewcommand{\tilde}[1]{\widetilde{#1}}
\renewcommand{\epsilon}{\varepsilon}
\renewcommand{\phi}{\varphi}
\newcommand{\T}{\mathbb{T}}
\renewcommand{\hat}[1]{\widehat{#1}}

\newcommand{\defeq}{:=}

\usepackage{xcolor}

\title{Freezing in the Infinite-Bin Model}
\author{
Bastien Mallein\thanks{Institut de Mathématiques de Toulouse, UMR 5219 - Université de Toulouse, France.\\ Email: \texttt{bastien.mallein@math.univ-toulouse.fr}} \and Sanjay Ramassamy\thanks{Universit\'e Paris-Saclay, CNRS, CEA, Institut de Physique Th\'eorique, France.\\ Email: \texttt{sanjay.ramassamy@ipht.fr}} \and Arvind Singh\thanks{Université Paris-Saclay, CNRS, Laboratoire de Mathématiques d'Orsay, France.\\ Email: \texttt{arvind.singh@universite-paris-saclay.fr}}}

\date{\today}

\begin{document}

\maketitle

\begin{abstract}
The infinite-bin model is a one-dimensional particle system on $\Z$ introduced by Foss and Konstantopoulos in relation with last passage percolation on complete directed acyclic graphs. In this model, at each integer time, a particle is selected at random according to its rank and produces a child at the location immediately to its right. In this article, we consider the limiting distribution of particles after an infinite number of branching events have occurred. Under mild assumptions, we prove that the event (called freezing) that a location contains only a finite number of balls satisfies a $0-1$ law and we provide various criteria to determine whether freezing occurs.
\end{abstract}

\section{Introduction}
\label{sec:introduction}

The infinite-bin model (IBM) is a discrete-time particle system on $\Z$ introduced by Foss and Konstantopoulos~\cite{FK}. In this process, at each integer time, a particle chosen according to its rank reproduces by creating a newborn particle one step to its right. Informally, it is constructed as a balls and bins scheme,  by placing a bin at each site of $z\in\Z$ with a certain number of balls, corresponding to the particles at initial time. We assume that all the bins at location $z$ large enough are initially empty, so it is possible to number all the balls in the system starting from the rightmost one. We define the dynamics of the infinite-bin model as follows. Let $\mu$ be a probability distribution on $\N \defeq \{1,2,\ldots\}$. At each integer time, the $k$th rightmost ball is selected with probability $\mu(k)$ and a ball is added to the bin immediately on its right. See Figure~\ref{fig:IBM} for an illustration of the evolution of the IBM.

\bigskip

The special case when $\mu$ is the uniform measure on $\{1,\ldots,N\}$ for some integer $N\geq1$ first appeared in \cite{AP}. It corresponds to a continuous-time branching random walk with selection. The case of general $\mu$ can be seen as a ranked-biased branching random walk with selection (see Remark~\ref{rem:genealogical}). The case when $\mu$ is a geometric distribution was studied by \cite{FK}, as it is coupled to a model of last-passage percolation on the complete directed acyclic graph where the edge weights may take values $1$ or $-\infty$. Considering more general laws for edge weights in the last-passage percolation model leads to the definition of max-growth systems \cite{FKMR}, generalizing the IBM with geometric distribution. All these particle systems fall into the class of processes with long memory, see e.g.~\cite{CFF}, or \cite{FoKoMaRa2023} for a review on infinite-bin models.

\bigskip

We introduce some notation needed to define the infinite-bin model formally. The state space of the IBM is defined as
\[
  \S \defeq \left\{ X : \Z \rightarrow \Z_+\cup \{\infty\} \text{ such that } \sup\{ j \in \Z : X(j) \neq 0 \} \in (-\infty,\infty) \right\}.
\]
An element $X\in \S$ represents a \emph{configuration of balls} where, for each $k \in \Z$, the bin at position $k$ contains $X(k)$ balls.

Let us point out that the total number of balls in a configuration $X \in \S$ may be finite or infinite (but cannot be zero). We even allow the number of balls in a given bin $k$ to be infinite i.e. $ X(k) = \infty$. For the dynamics of the infinite bin model, such a bin will represent an impassable barrier (the configuration on the left of this bin will not matter and remain unchanged forever). A configuration $X \in \S$ is said to be
\begin{itemize}
\item \emph{locally finite} if $X(k) < \infty$ for all $k\in \Z$,
\item \emph{bounded} if $\sup_{k\in \Z} X(k) < \infty$.
\end{itemize}

Given $\mu$ a probability distribution on $\N$ and $X_0 \in \S$ an initial configuration, an infinite-bin model $X$ with reproduction law $\mu$ starting from $X_0$ is constructed as an $\S$-valued Markov process in the following fashion. Given $(\xi_n, n \in \N)$ i.i.d. random variables with law $\mu$, at each integer time $n$, $X_n$ is constructed from $X_{n-1}$ by adding a new ball to the bin immediately to the right of the bin containing the $\xi_{n}$th rightmost ball. In other words, writing $J_{n-1}$ for the index of the bin containing the $\xi_{n}$th rightmost ball at time $n-1$, defined by
\[
  \sum_{j = J_{n-1}}^\infty X_{n-1}(j) \geq \xi_{n} > \sum_{j=J_{n-1} +1}^\infty X_{n-1}(j),
\]
we set\footnote{If such an index does not exist, which is the case if the configuration $X_{n-1}$ contains fewer than $\xi_{n}$ balls, we just set $X_{n}=X_{n-1}$.} $X_{n}(\cdot) = X_{n-1}(\cdot) + \ind{\cdot = J_{n-1}+1}$.

In recent years, all the articles studying infinite-bin models, such as \cite{FK,FZ,CR,MR17,MR16, Terlat}, focused on studying the displacement of the front at time $n$, i.e. the position of the rightmost occupied bin in $X_n$, as well as the asymptotic behavior of this quantity as $n \to \infty$. In these articles, it is therefore natural to study the infinite-bin model ``seen from the front'', and to describe the statistical properties of the largest non-empty bins.

In contrast, we focus here on a different question and study instead how the number of balls in a given bin grows as time passes. Obviously, for all $k \in \Z$, the number $X_n(k)$ of balls in the bin at position $k$ at time $n$ is non-decreasing in $n$ (since balls are never removed from bins). Hence we can define the random variables
\[
  X_\infty(k) = \lim_{n \to \infty}\uparrow X_n(k) \in \Z_+ \cup \{\infty\} \quad \text{a.s.}
\]
We call the event $\{X_\infty(k) < \infty\}$ the \emph{freezing} of bin $k$, as it means that, on this event, there is some finite (random) time after which no ball is ever added to bin $k$. This paper aims to describe conditions on the distribution $\mu$ and the initial configuration $X_0$ for which freezing occurs.

\medskip

As an appetizer, let us first note that the question above is readily answered when  $\mu$ has a first moment:
\begin{proposition}
\label{prop:strongFreezing}
Suppose that the distribution $\mu$ satisfies $\sum n\mu(n) < \infty.$ Then for any initial configuration $X_0 \in \S$ and any $k\in\Z$ such that $X_0(k) < \infty$, we have $X_\infty(k) < \infty$ a.s.
\end{proposition}
This result is an easy consequence of the Borel-Cantelli lemma. It also follows directly from the existence of renovation events for the IBM first introduced by Foss and Konstantopoulos in \cite{FK} and extended in \cite{CR, FZ}.

\medskip

Taking a different direction, instead of making assumptions on $\mu$, we can rather make assumptions on the initial configuration $X_0$ which will ensure freezing for any distribution $\mu$.
\begin{theorem}\label{thm:unifbounded}
Let $X_0 \in \S$ be a bounded configuration. For any distribution $\mu$, we have $X_\infty(k) < \infty$ a.s. for all $k\in \Z$.
\end{theorem}
One may wish to replace the ``bounded'' assumption in the above theorem by  ``locally finite''. Unfortunately, this is not possible as freezing is not always guaranteed in that case. In fact, even the definition of freezing itself may be ambiguous because the event $\{ X_\infty(k) < \infty\}$ need not be an a.s. deterministic event. The apparently simple question of whether a $0-1$ law holds for the freezing of a bin turns out to be surprisingly tricky. We give a partial answer, that requires a monotonicity assumption on $\mu$.
\begin{theorem}
\label{thm:01law}
Assume that the probability distribution $\mu$ satisfies
\begin{equation}
  \label{eqn:monotonous}
  \hbox{$(\mu(n), n \geq 1)$ is non-increasing.}
\end{equation}
Then, for any initial configuration $X_0 \in \S$, we have
\[
  \P(X_\infty(k) < \infty) \in \{0,1\} \quad \hbox{for all $k\in \Z$.}
\]
\end{theorem}

\begin{remark}
Assumption \eqref{eqn:monotonous} could be replaced by $(\mu(n), n\geq 1)$ ultimately non-increasing up to simple modifications in the proof. However, we do not know whether the assumption can be completely dropped.
\end{remark}

\medskip

One of the main ways in which bins can have a growing number of balls as $n \to \infty$ stems from the ``cascade effect'': when a bin becomes large, it has a greater chance to be hit by the $(\xi_n)$ and therefore it increases the rate of growth of the bin to it right (while decreasing its own growth rate). This reinforcement scheme can lead to non-freezing of bins for particular initial configurations. As we will see below, it can even be the case that some bins grow to infinity while others freeze a.s. To study the different freezing scenarii in detail, it is useful to consider first the extremal configuration $\widehat{X}_0$ consisting of a single infinite bin (i.e. a barrier) at $0$ while all the other bins are empty:
\begin{equation}\label{def:widehatX}
  \widehat{X}_0(k) \defeq
  \begin{cases}
    \infty & \text{ if } k = 0,\\
    0 & \text{ otherwise}.
  \end{cases}
\end{equation}
We will denote by $(\widehat{X}_n, n \geq 0)$ the IBM process starting from $\hat{X}_0$.

\begin{definition}
\label{def:classification}
For $d \in \N \cup \{\infty\}$, we say that the probability distribution $\mu$ is of \emph{type $d$} if
\[
  d =\inf\{k\geq 1 \,:\, \widehat{X}_\infty(k) < +\infty \} \quad \text{a.s.}
\]
We say that $\mu$ is of \emph{finite type} (resp. of \emph{infinite type}) if $d <+\infty$ (resp. $d =+\infty$).
\end{definition}

In view of Proposition~\ref{prop:strongFreezing}, it is easy to check that $\mu$ is of type $1$ if and only if it has a first moment.\footnote{Proposition~\ref{prop:strongFreezing} shows that when $\mu$ has a first moment, it is of type $1$. To prove the converse result, we simply notice that, when $\mu$ does not have a first moment, the Borel-Cantelli lemma for independent events ensures that $\xi_n \geq n$ i.o. and therefore $\widehat{X}_\infty(1) = \infty$ hence $\mu$ is not of type $1$. }

Let us point out that the type of a distribution may not necessarily be well-defined without a $0-1$ law on the freezing of bins. Yet, it will be proved in Corollary~\ref{cor:cascadefinite} that the inclusion $\{\hat X_\infty(k) = \infty\} \subset \{ \hat X_\infty(k-1) = \infty \}$ holds a.s for any distribution $\mu$ and any $k\geq1$. Combining this result with Theorem~\ref{thm:01law}, we conclude that
\begin{corollary}
If  $(\mu(n), n \geq 1)$ is non-increasing, then $\mu$ has a well-defined deterministic type $d \in \N \cup \{\infty\}$ and the set of non-freezing bins is the interval:
\[
  \left\{ k\in\Z,\; \hat{X}_\infty(k) = \infty \right\} = \llbracket 0, d - 1 \rrbracket \quad \text{a.s.}
\]
\end{corollary}

Finding the exact type of a general distribution $\mu$ seems a delicate question. However, with the additional assumption that the distribution has regular variation, one can provide an explicit integral criterion to characterize the type of $\mu$. Recall that a function $\gamma : \Z_+ \to (0,\infty)$ is said to have regular variation with index $\alpha$ (at infinity) if
\[
  \lim_{n\to\infty}\frac{\gamma( \floor{u n})}{\gamma( n)} = u^\alpha\quad\hbox{for all $u>0$.}
\]
When $\alpha = 0$, the function is said to be slowly varying. The next result shows, in particular, that there exist distributions for any type $d \in \N \cup \{+\infty\}$.

\begin{theorem}
\label{thm:typeK}
Assume that $(\mu(n), n \geq 1)$ is  regularly varying at infinity with index $-\alpha$ for some $\alpha \geq 1$. Let $\bar{\mu}(j) = \mu(\llbracket j, +\infty\llbracket)$ denote its tail. Then $\mu$ has a well-defined type $d$ given by
\[
d = \inf\Big\{ h \in \N: \; \sum_{j} \bar{\mu}(j)^h < \infty\Big\}.
\]
In particular, $\mu$ is of finite type if $\alpha > 1$ and of infinite type if $\alpha=1$.
\end{theorem}

\begin{remark}
\begin{enumerate}
\item The criterion above for type $d=1$ corresponds to the first moment condition on $\mu$, which we know to be a necessary and sufficient condition, regardless of the regular variations of $\mu$ at $\infty$.
\item If $(\mu(n), n\geq 1)$ is regularly varying with index $-\alpha$, then by Karamata's theorem (c.f. for instance \cite[Chapter VIII.9]{Feller71}), its tail is also regularly varying at infinity with index $-\alpha + 1$. Thus, there exists a slowly varying function $L$ such that $\bar{\mu}(n) = n^{-\alpha + 1} L(n)$. The type $d$ of $\mu$ is then given by the formula:
\[
d = \begin{cases}
+\infty & \hbox{if $\alpha = 1$,}\\
\lceil \frac{1}{\alpha - 1} \rceil + 1 & \hbox{if $\frac{1}{\alpha- 1}$ is an integer and $\sum_j \frac{L(j)^{\frac{1}{\alpha- 1}}}{j} = \infty$,}\\
\lceil \frac{1}{\alpha - 1} \rceil & \hbox{otherwise.}
\end{cases}
\]
\item The proof of Theorem~\ref{thm:typeK} gives extra information on the asymptotic rate at which each bin grows. For a distribution $\mu$ as in Theorem~\ref{thm:typeK} with finite type $d$, we find that, for any $k \in \llbracket 1, d \llbracket$,
\[
  \hat{X}_n(k) \underset{n\to\infty}{\sim} C_{k,\alpha} \sum_{j \leq n} \bar{\mu}(j)^k \quad \hbox{a.s.}
\]
where $C_{k,\alpha}$ are explicit positive constants depending only on $k$ and $\alpha$.
\end{enumerate}
\end{remark}

By definition, the type of a distribution $\mu$ is the number of bins (plus $1$) that do not freeze when starting from the barrier configuration $\widehat{X}_0$. If $\mu$ has finite type, this roughly tells us that the number of consecutive bins in a ``cascade'' must be finite. But then, this means that such cascades should not be able to propagate from $-\infty$ so we should expect almost sure freezing of all bins for every locally finite initial configuration. Our last result extends Theorem~\ref{thm:unifbounded} to locally finite configurations and relates the finiteness of the type with the existence of non-freezing configurations.

\begin{theorem}~
\label{thm:finitevinfinite}
\begin{enumerate}
\item If $\mu$ is of finite type and $(\mu(n), n \geq 1)$ is non-increasing, then, for any initial locally finite configuration $X_0 \in \S$, freezing occurs a.s.:
$$X_\infty(k) < \infty \quad\hbox{a.s. for all $k \in \Z$.}$$
\item If $\mu$ is of infinite type, then there exists a locally finite configuration $X_0 \in \S$ such that no bin freezes:
$$X_\infty(k) = \infty \quad\hbox{a.s. for all $k \in \Z$.}$$
\end{enumerate}
\end{theorem}

The rest of the article is organized as follows. In the next section, we study the genealogical structure of the IBM, showing in particular that each ball has a finite progeny a.s. From this fact, we deduce Theorem~\ref{thm:unifbounded}, and carry on proving Item $2.$ of Theorem~\ref{thm:finitevinfinite}.

In Section~\ref{sec:coupling}, we describe a particular coupling of the IBM starting from different initial configurations under Assumption \eqref{eqn:monotonous}. This coupling implies the negative correlation for the events of adding a ball in the same bin at two different times and it enables us to prove  Theorem~\ref{thm:01law} as well as Item $1.$ of Theorem~\ref{thm:finitevinfinite}.

Finally, the last section is devoted to the study of the type of $\mu$ under the regularly varying assumption. We provide here the proof of Theorem~\ref{thm:typeK}, by a simple concentration result for sums of independent Bernoulli variables.

\section{General properties of infinite-bin models}
\label{sec:preliminaries}

Before studying properties of the infinite-bin model dynamics, we introduce a few extra notations that will allow us to describe in greater details the local dynamic of the model.

Given an infinite-bin model configuration $X \in \S$, we define its \emph{rightmost barrier} as
\begin{equation}\label{def:deltabarrier}
\delta(X) \defeq \sup\{j\in \Z : X(j) = +\infty\} \;\in \Z\cup\{-\infty\}
\end{equation}
with the convention $\delta(X) = \sup \emptyset = -\infty$ when $X$ is locally finite. The \emph{front} of $X$ is the location of its rightmost non-empty bin (which is well-defined by definition of a configuration), i.e.
$$F(X) \defeq \max\{ j \in \Z : X(j) \neq 0\} \in \Z.$$
For all $k \in \N$, we also define $B(X,k)$ as the location of the $k$th rightmost ball in $X$, \emph{i.e.}
\[
  B(X,k) \defeq \sup\{ j \in \Z : \sum_{i=j}^\infty X(i) \geq k \},
\]
again with $\sup \emptyset =  -\infty$. We observe that $F(X) = B(X,1)$ and $\delta(X) = \lim_{k\to\infty} B(X,k)$.

\begin{remark}
For all $X \in \S$ and $k \in \N$, the quantity $B(X, k)$ is well-defined irrespectively of the relative ordering of the balls inside each bin. In other words, the dynamic of the infinite-bin model does not require to fix an ordering of the balls. However, it will be convenient to specify an ordering for the balls in each bin, in order to introduce a genealogical structure to the particle system.
\end{remark}

To give a formal definition of the infinite-bin model dynamics, we introduce the operator $\Phi_\xi : \S \to \S$ that maps a configuration to the new configuration obtained by adding a single new ball immediately to the bin on the right of the $k$th rightmost ball:
\begin{equation}\label{eq:operatorPhi}
  \Phi_\xi(X) \defeq \begin{cases}
     (X(j) + \ind{j = B(X,\xi)+1}, j \in \Z) & \text{ if } B(X,\xi) > -\infty,\\
     X & \text{ otherwise.}
  \end{cases}
\end{equation}
Given a probability distribution $\mu$ and an initial configuration $X_0 \in \S$, the infinite-bin model $(X_n,\, n\geq 0)$ with distribution $\mu$ starting from $X_0 \in \S$ can be constructed through the following recursion equation, for $n\geq 0$
\begin{equation}
  \label{eqn:ibm}
  X_{n+1} \defeq \Phi_{\xi_{n+1}} (X_n),
\end{equation}
where $(\xi_n, n \geq 1)$ is an i.i.d. sequence of random variables with law $\mu$.

\begin{figure}
   \begin{center}
   \subfloat[{Configuration $X$}]{\includegraphics[height=2.5cm]{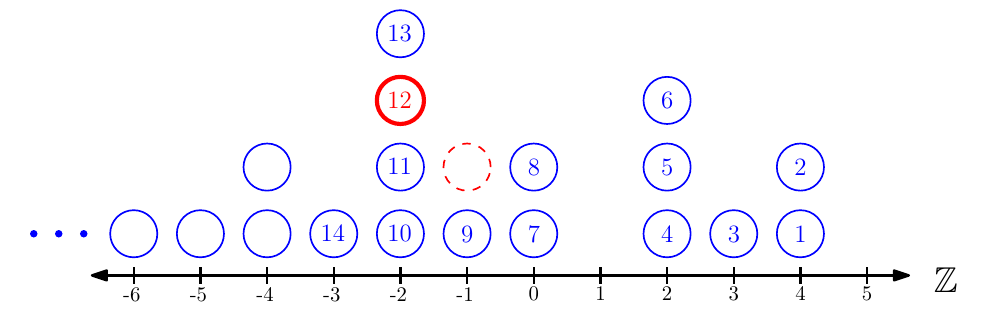}}
   \quad
   \subfloat[{Configuration $\Phi_{\xi}(X)$ with $\xi = 12$.}]{\includegraphics[height=2.5cm]{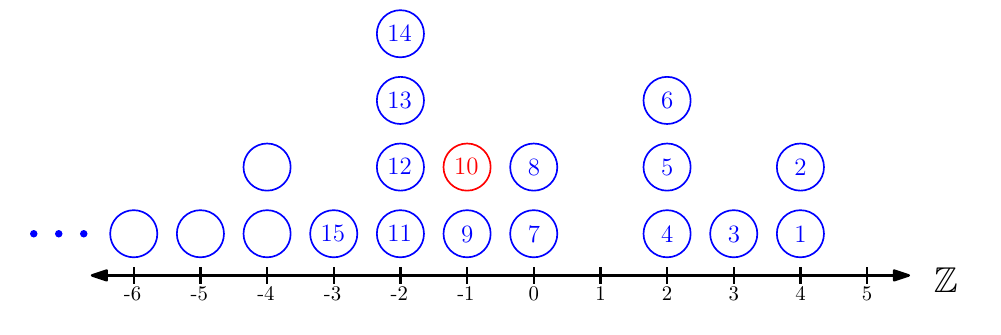}}
   \end{center}
 \caption{\label{fig:IBM}Illustration of the dynamics of the IBM: here, we have a configuration $X$ of balls with $F(X) = 4$. The result of sampling $\xi = 12$ creates a new ball at location $-1$, height $2$, and rank $10$. All the balls with previous rank $\geq 10$ in $X$ have their rank increased by one in $\Phi_{\xi}(X)$.}
\end{figure}

In the rest of this article, we always picture balls as stacked vertically in their bin, and ranked from bottom to top. When a new ball is added, it is placed on the top of the previous balls in the bin. With this ordering, for each ball in $X$, we can speak about
\begin{itemize}
\item its \emph{rank} which corresponds to its index when all balls are ordered starting from the front, from right to left, and then from bottom to top inside each bin.
\item its \emph{location} which is the index/position of the bin it belongs to.
\item its \emph{height} which corresponds to its index inside the bin with the balls ordered from bottom to top.
\end{itemize}
Let us note that, by definition, for each $\xi \in \N$ the location and heights of the balls already present in $X$ are unchanged in the new configuration $\Phi_\xi(X)$. However, the ranks of the balls in $\Phi_\xi(X)$ are either unchanged or increased by one. See Figure~\ref{fig:IBM} for an illustration of the procedure.

\subsection{Genealogical structure}

The ranking introduced above on all the balls in the infinite-bin model allows us to define a genealogical structure for the particle system. Given the infinite-bin model defined by \eqref{eqn:ibm}, for all $n \in \N$, we state that the ball added at time $n$ is a child of the ball ranked $\xi_{n}$th in the configuration $X_{n-1}$.

\begin{remark}
\label{rem:genealogical}
Adding this genealogical structure to the infinite-bin model turns this process into a rank-based branching particle system. In particular, given $(N_t, t \geq 0)$ an independent Poisson process with parameter $1$, the evolution of the continuous-time particle system $(X_{N_t}, t \geq 0)$ can be described as follows. At all time $t$, every particle in the system is ranked from right to left. For all $k \in \N$, the $k$th rightmost ball gives birth to a child at distance $1$ from its current position, independently of any other particle, at rate $\mu(k)$. This family of branching processes has been the subject of a large literature, in particular in its interpretation as a branching process with selection \cite{BrunetDerrida,BerardGouere,Maillard,Tourniaire}.
\end{remark}

With the genealogical structure defined above, we observe that the total number of children made by a given ball in the infinite-bin model is on average finite, without requiring any assumption on $\mu$.
\begin{lemma}
\label{lem:finiteOffspring}Fix a configuration $X_0 \in \S$ and $k\in \N$ such that $k \leq \sum_{j \in \Z} X_0(j)$ (i.e. the $k$th ranked ball exists). Then, writing $R_k$ the total number of children at all times of the $k$th rightmost ball in $X_0$, we have
\[
  \E(R_k) \leq \sum_{\ell=k}^\infty  \frac{\mu(\ell)}{\mu(\llbracket 1,\ell\rrbracket)} < \infty
\]
with the convention $0/0 = 0$ when $\mu(\llbracket1,\ell\rrbracket) = 0$.
\end{lemma}

\begin{proof}
Without loss of generality, we assume that the $k$th rightmost ball, written~$\partial$, is in the bin of index $0$. For all $n \in \N$, we denote by $M_n \defeq \sum_{j \geq 1} X_n(j)$ the total number of balls to the right of bin $0$ at time $n$ and by $N_n$ the total number of balls in bin $0$ at time $n$. We also denote by $h \defeq k -M_0$ the height at time $0$ (and at all times) of the ball $\partial$ inside bin $0$. Observe that at time $n$, the ball $\partial$ is the $(M_n+h)$th rightmost ball.

At each time $n \in \N$, one of two things might happen:
\begin{itemize}
  \item[-] With probability $\mu(\llbracket1,M_n+N_n\rrbracket)$, a new ball is added in a bin of positive index, and thus $M_{n+1} = M_n +1$. In that case, the probability that this new ball is an offspring of the ball $\partial$ is given by
  \begin{equation}
    \label{eqn:ub1}
    \frac{\mu(M_n + h)}{\mu(\llbracket 1,M_n+N_n\rrbracket)} \leq \frac{\mu(M_n + h)}{\mu(\llbracket 1,M_n+h\rrbracket)}.
  \end{equation}
  \item[-] With probability $1 -  \mu(\llbracket 1,M_n+N_n\rrbracket)$, a new ball is added in a bin of non-positive index (or is not added if $\xi_{n+1}$ is too large). In this case, $N_n$ might increase, or not, by $1$.
\end{itemize}
We now observe that the total number of children of $\partial$ is given by the sum
\[
  R_k = \sum_{n \geq 0} \ind{M_{n+1} = M_n + 1} \ind{\text{the new ball is a child of $\partial$}}
\]
We partition this sum according to the value $m$ taken by $M_n$ at each time $n$, combining with \eqref{eqn:ub1} and the Fubini theorem to obtain that
\begin{align*}
  \E(R_k) & = \sum_{n \geq 0} \P(M_{n+1} = M_n + 1, \text{the new ball is a child of $\partial$})\\
  & = \sum_{n \geq 0} \sum_{m \geq M_0} \P(M_n = m) \P(M_{n+1} = M_n+1 , \text{the new ball is a child of $\partial$}|M_n = m)\\
  & \leq \sum_{m=M_0}^\infty \frac{\mu(m+h)}{\mu(\llbracket 1,m+h \rrbracket)} = \sum_{\ell=k}^\infty  \frac{\mu(\ell)}{\mu(\llbracket 1,\ell\rrbracket)}.
\end{align*}
Finally, since $\sum_{j=1}^\infty \mu(j)=1$, we have
\[
\tfrac{\mu(\ell)}{\mu(\llbracket1,\ell\rrbracket)}\underset{\ell\rightarrow\infty}{\sim}\mu(\ell),
\]
hence
\[
\sum_{\ell=k}^\infty  \frac{\mu(\ell)}{\mu(\llbracket 1,\ell\rrbracket)} <  \infty
\]
which completes the proof.
\end{proof}

A direct consequence of this result is that the set of non-frozen bins forms an interval.
\begin{corollary} \label{cor:cascadefinite}
Let $k \in \Z$ and $X_0 \in \S$ such that $X_0(k) < \infty$. We have
$$\{X_\infty(k) = \infty\} \subset \{X_\infty(k-1) = \infty\}\quad\hbox{a.s.}$$
\end{corollary}

\begin{proof}
We observe that on the event $\{X_\infty(k-1) < \infty\} \cap \{X_\infty(k) = \infty\}$ at least one of the balls in the bin $k-1$ has had an infinite number of offspring, which is a.s. impossible given the bound obtained in Lemma~\ref{lem:finiteOffspring}, and as there is only a countable number of balls involved in the process.
\end{proof}

\subsection{Freezing for bounded configuration}

We prove here Theorem~\ref{thm:unifbounded} which states that for any probability distribution $\mu$, if the initial configuration $X_0$ of the infinite-bin model is bounded, then freezing occurs almost surely for all bins. It is a consequence of the following slightly stronger result.

\begin{proposition}\label{prop:bounded} For any $A\in \N$, there exists a constant $C_A <\infty$ depending only on $A$ and the distribution $(\mu(n), n\geq 0)$ such that, for any initial configuration $X_0 \in \S$ satisfying
$$\sup_{k \in \Z} X_0(k) \leq A,$$
it holds that
\begin{equation}\label{eq:unifboundE}
\sup_{k \in \Z} \E(X_\infty(k)) < C_A.
\end{equation}
In particular, if $X_0 \in \S$ is a bounded configuration, then $X_\infty(k) < \infty$ a.s. for all $k\in \Z$ so every bin freezes.
\end{proposition}

\begin{proof}
Fix a configuration $X_0 \in \S$ such that $\sup_{k \in \Z} X_0(k) \leq A$ and fix $k\in \Z$. We define $((T_n,r_n), n \geq 0)$ as the sequence of times at which a new ball is added to a bin strictly to the right of the bin $k-1$, together with the rank of its parent. Fixing $T_0\defeq 0$, we set
\[
  T_{n+1} \defeq \inf\big\{m > T_n : \xi_m \leq \sum_{i=k-1}^\infty X_{m-1}(i) \big\} \quad \text{and} \quad r_n \defeq \xi_{T_n}.
\]
Denoting $M_n$ the total number of balls in the bins strictly to the right of $k-1$ at time $T_n$, we have that
$M_n \defeq \sum_{i=k}^\infty X_{T_n}(i) = M_0 + n$, and that $r_n$ has law $\mu$ conditioned on being at most equal to $M_n + X_{T_{n}}(k-1)$.

For any $n_0 \in \N$, we have the upper bound
\begin{equation}\label{eq:upbound1}
X_{\infty}(k) = X_0(k) + \sum_{n=0}^{+\infty} \ind{r_n > M_n}  \leq A + n_0 + \sum_{n=n_0}^{+\infty} \ind{r_n > M_n}.
\end{equation}
We recall the notation $\bar{\mu}(n) = \mu(\llbracket n,\infty\llbracket)$ for the tail of $\mu$. We fix $n_0$ large enough such that
$$
\alpha  \defeq \frac{\bar{\mu}(n_0)}{\mu(\llbracket1,n_0\rrbracket)} < 1.
$$
Taking the expectation in \eqref{eq:upbound1} conditionally on $\mathcal{H} = \sigma(X_{T_n}(k-1), n \geq 0)$, we find that
\begin{align*}
\E( X_\infty(k) \;|\; \mathcal{H} )  & \leq A + n_0 + \sum_{n=n_0}^{+\infty} \P(r_n > M_n \;| \; \mathcal{H})\\
& =  A  + n_0 + \sum_{n=n_0}^{+\infty} \frac{\mu(\llbracket M_n+1,M_n+X_{T_n}(k-1)\rrbracket)}{\mu(\llbracket 1,M_n+X_{T_n}(k-1)\rrbracket)}\\
& \leq A  + n_0 +  \frac{\sum_{n=n_0}^\infty \mu(\llbracket M_0 + n + 1,M_0 + n + X_{\infty}(k-1)\rrbracket)}{\mu(\llbracket 1,n_0 \rrbracket)}  \\
& \leq A  + n_0 +  \frac{\bar{\mu}(M_0 + n_0 + 1) X_\infty(k-1)}{\mu(\llbracket 1,n_0 \rrbracket)}  \\
& \leq A  + n_0 + \alpha  X_\infty(k-1).
\end{align*}
Taking the expectation on both sides of the inequality above, we obtain a recursion inequality valid for all $k\in \Z$,
\begin{equation}\label{eq:recbounded}
\E( X_\infty(k) ) \leq A  + n_0 + \alpha  \E(X_\infty(k-1)).
\end{equation}
Let $\ell\leq0$ and consider now the truncated configuration $X^{(\ell)}_0 = X_0 \ind{. \geq \ell}$ where all the balls in the bins on the left of bin $\ell$ have been removed. We take $\ell$ to be small enough so that $X^{(\ell)}_0$ is non-empty. We denote by $X^{(\ell)}$ the IBM process starting from $X^{(\ell)}_0$, using the same sequence $(\xi_n)$ to construct $X^{(\ell)}_n$ recursively by \eqref{eqn:ibm}. By definition of the dynamics, no ball can ever be added to bin $\ell$  so
\[
  \E(X^{(\ell)}_\infty(\ell))  = X^{(\ell)}_0(\ell) \leq A.
\]
In view of \eqref{eq:recbounded}, it follows by induction that, for all $k\geq \ell$, we have
$$
\E(X^{(\ell)}_\infty(k))  \leq C_A \defeq \frac{A + n_0}{1-\alpha}.
$$
Finally, for each fixed $k$ and $n$, we have $\lim_{\ell \to -\infty} X_n^{(\ell)}(k) = X_n(k)$ a.s. Indeed, we have either $\sum_{k \in \Z} X_0(k) < \infty$, then $X_0 = X_0^\ell$ for $\ell$ large enough, or $\sum_{k \in \Z} X_0(k) = \infty$, then a.s., for $-\ell$ large enough we have $\max_{j \leq n} \xi_j \leq \sum_{k = \ell}^\infty X_0(k)$, therefore $X^{(\ell)}_n(k) = X_n(k)$.  As a result, by Fatou's Lemma, for any $k\in\Z$,
\[
\E(X_n(k)) \leq \liminf_{l \to -\infty} \E(X_n^{(\ell)}(k))  \leq C_A.
\]
By monotone convergence, we conclude that, for any $k\in\Z$,
\[
  \E(X_\infty(k)) = \lim_{n \to +\infty} \E(X_n(k)) \leq C_A.
\]
This completes the proof of \eqref{eq:unifboundE} and also of Theorem~\ref{thm:unifbounded} since a random variable with finite expectation is finite a.s.
\end{proof}

\subsection{Proof of Item \texorpdfstring{$2.$}{2.} of Theorem~\ref{thm:finitevinfinite}}

We assume here that $\mu$ is of infinite type and we explain how to construct an initial configuration $X^*_0 \in \S$ such that
\begin{equation}\label{eq:confXstar}
X^*_0(k) < \infty \quad \hbox{and} \quad X^*_\infty(k) = \infty \hbox{ a.s.}\quad \hbox{for all $k\in\Z$.}
\end{equation}

\begin{proof}[Proof of Item $2.$ of Theorem~\ref{thm:finitevinfinite}]
Some technicalities in the argument below appear because the support $T_\mu\subset\mathbb N$ of the distribution $\mu$ is not necessarily equal to $\N$. However, because $\mu$ is of infinite type, its support $T_\mu$ must be infinite. Recall that we write $F(X)$ for the position of the front of the configuration $X \in \S$, we then denote by $\S_\mu$ the set of all configurations $X\in \S$ satisfying the following conditions:
\begin{enumerate}
\item $X(0)=\infty$,
\item for all $k \in \llbracket 1, F(X) \rrbracket$ we have $X(k)\in T_\mu$ and $X(k)\geq X(k+1)$.
\end{enumerate}
Recall that $\widehat{X}_0$ defined by \eqref{def:widehatX} denotes the configuration with a single infinite barrier at $0$. It is not hard to see that any configuration $X\in \S_\mu$ is a reachable state for the IBM($\mu$) starting from $\widehat{X}_0$. Thus, by definition of a distribution of infinite type, for any initial condition $X_0\in S_\mu$, the process will eventually yield $X_\infty(k) = \infty$ a.s. for all $k \geq 1$. Additionally, if we define, for $h \in \N$,
\[
  X^{(h)}_0(k) \defeq
  \begin{cases}
    X_0(k) &\text{ if } k \neq 0,\\
    h &\text{ otherwise,}
  \end{cases}
\]
then it is straightforward from the definition of the IBM that, for all fixed $n, k \in \N$ one has
$$\lim_{h \to \infty} X^{(h)}_n(k) = X_n(k)\quad\hbox{a.s.}$$
As a result, using the monotonicity of the process  we conclude that, for all $k\geq 1$,
$$\lim_{h \to \infty} X^{(h)}_\infty(k) = X_\infty(k) = \infty \quad\hbox{a.s.}$$

Thanks to the above observation, we now construct a locally finite initial configuration $X^*_0$ satisfying \eqref{eq:confXstar}. We start by setting $X^*_0(k) \defeq 0$ for all $k \geq 0$, and we define recursively $X^*_0(-k)$ for $k > 0$ as follows. For all $k \in \N$, there exists $n_k \geq 1$ such that the infinite-bin model with the starting configuration $Y_0$ defined by
$$
Y_0(j) \defeq
\begin{cases}
r  & \hbox{if $j = -k$,}\\
X^*_0(j) & \hbox{if $j \in \llbracket -k + 1, -1\rrbracket$,}\\
0 & \hbox{otherwise,}
\end{cases}
$$
satisfies $\P( \min_{j \in \llbracket -k,k\rrbracket} Y_\infty(j) \leq k ) \leq 2^{-k}$ when $r \geq n_k$.
We assign to $X^*_0(-k)$ any value in $T_\mu$ that is greater than $\max(n_k , X^*_0(-k+1))$.

Using again the monotonicity of the process, it is a direct consequence of this construction that $X^*_\infty(k) = \infty$ a.s. for all $k \in \Z$.
\end{proof}

\section{Coupling of IBMs for different initial configurations}
\label{sec:coupling}

In this section, we prove Theorem~\ref{thm:01law} and Item $1.$ of Theorem~\ref{thm:finitevinfinite}. Both results make repetitive use of couplings of IBMs starting from different initial configurations. We start with the following easy lemma which relates the evolution of two IBMs under the usual coupling.

\begin{lemma}\label{lemma:couplage1} Let $X_0 , \tilde{X}_0 \in \S$ denote two initial configurations with the same position $\delta$ for their rightmost infinite barrier (if it exists):
$$
\delta \defeq \delta(X_0) = \delta(\tilde{X}_0) \in \Z\cup \{-\infty\}.
$$
Fix $j_0 > \delta$ and assume that the following condition holds true at time $n=0$
\begin{equation}
\hbox{for all $j \in \rrbracket \delta,  j_0 \rrbracket$,}\quad\tilde{X}_n(j) = X_n(j) \quad\hbox{and}\quad \sum_{k =j}^{\infty} X_n(k) = \sum_{k=  j}^{\infty} \tilde{X}_n(k).\tag{H1}\label{hypH1}
\end{equation}
Let $(X_n)$ and $(\tilde{X}_n)$ be defined by the recursion equation \eqref{eqn:ibm} using the same sequence $(\xi_n)$. Then \eqref{hypH1} holds true for all times $n$. Furthermore, we have, for all $n\geq 0$ and all $j\in \rrbracket \delta, j_0 + 1\rrbracket$,
\begin{equation}\label{incH1}
\tilde{X}_{n+1} (j) = \tilde{X}_{n} (j) + 1 \quad \Longleftrightarrow  \quad X_{n+1} (j) = X_{n} (j) + 1.\qquad\hbox{a.s.}
\end{equation}
\end{lemma}
\begin{proof}
We prove the result by induction on $n$. It is true for $n=0$ by assumption. Suppose now that \eqref{hypH1} holds for a given $n\geq 0$ and consider $\xi_{n+1}$. There are two cases:
\begin{itemize}
\item If $\xi_{n+1} \leq \sum_{k =j_0}^{\infty} X_n(k) = \sum_{k=  j_0}^{\infty} \tilde{X}_n(k)$, then we add a ball to $X_n$ and a ball in $\tilde{X}_n$, possibly at different locations, but always in a bin strictly on the right of $j_0$ in both cases. Therefore, \eqref{hypH1} holds for $n+1$ (as well as \eqref{incH1} because no ball is added to bin $k$).
\item If $\xi_{n+1} > \sum_{k =j_0}^{\infty} X_n(k) = \sum_{k=  j_0}^{\infty} \tilde{X}_n(k)$, then thanks to \eqref{hypH1}, the ball ranked $\xi_{n+1}$  in $X_n$ is at the same location and same height as the ball ranked $\xi_{n+1}$  in $\tilde{X}_n$. Thus, the ball added at time $n+1$  is at the same location in $X_{n+1}$ and $\tilde{X}_{n+1}$ hence \eqref{hypH1} and \eqref{incH1} still hold.
\end{itemize}
\end{proof}

We can improve on the previous lemma by coupling in a larger class of initial configurations if we assume that the distribution $\mu$ is monotonic.

 \begin{lemma}\label{lemma:couplage2} Assume that $(\mu(n), n \geq 1)$ is non-increasing. Let $X_0 , \tilde{X}_0 \in \S$ denote two initial configurations with the same position $\delta$ for their rightmost infinite barrier (if it exists):
$$
\delta \defeq \delta(X_0) = \delta(\tilde{X}_0) \in \Z\cup \{-\infty\}.
$$
Fix $j_0 > \delta$ and assume that the following condition holds true at time $n=0$
\begin{equation}
\hbox{for all $j \in \rrbracket \delta,  j_0 \rrbracket$,}\quad\tilde{X}_n(j) \leq X_n(j) \quad\hbox{and}\quad \sum_{k =j}^{\infty} X_n(k) \leq \sum_{k=  j}^{\infty} \tilde{X}_n(k).\tag{H2}\label{hypH2}
\end{equation}
 then, there exists a coupling of two infinite-bin models $(X_n)$ and $(\tilde{X}_n)$ with distribution $\mu$, starting  respectively from $X_0$ and $\tilde{X}_0$ such that \eqref{hypH2} holds for all times $n$. Furthermore, this coupling also satisfies that, for $n\geq 0$ and all $j \in \rrbracket \delta, j_0 + 1\rrbracket$,
\begin{equation}\label{incH2}
\tilde{X}_{n+1} (j) = \tilde{X}_{n} (j) + 1 \quad \Longrightarrow  \quad X_{n+1} (j) = X_{n} (j) + 1.\qquad\hbox{a.s.}
\end{equation}
\end{lemma}
See Figure~\ref{fig:couplage} for an illustration of two configurations $X$ and $\tilde{X}$ that satisfy assumption \eqref{hypH2}.
\begin{remark} If $X_0$ and $\tilde{X}_0$ satisfy \eqref{hypH2}, then the same holds true after applying either of the following two elementary operations: (1) We add a ball in $\tilde{X}_0$ at a location $j > j_0$. (2) We move a ball of $\tilde{X}_0$ from a location $j \leq j_0$ to a new location $j' > j_0$.
\end{remark}

\begin{figure}
\begin{center}
\includegraphics[height=5cm]{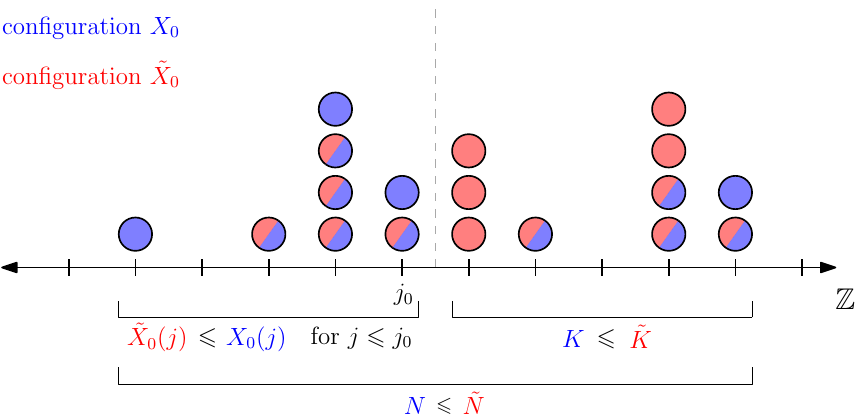}
\end{center}
\caption{\label{fig:couplage}Two configurations $X_0$ (in blue) and $\tilde{X}_0$ (in red) that satisfy \eqref{hypH2}. There are more red balls than blue balls located in $\llbracket j_0 + 1,+\infty\llbracket$ ($K = 5 \leq 9 = \tilde{K}$) and also more red balls than blue balls in total  ($N = 13 \leq 14 = \tilde{N}$) but there are at least as many blue balls than red balls in every bin located at $j < j_0$.}
\end{figure}

\begin{proof} We just need to construct the coupling for the first step and then we can proceed by induction afterward thanks to the Markov property. Let $$N \defeq \sum_{j > \delta} X_0(j) \quad\hbox{and} \quad \tilde{N} \defeq \sum_{j > \delta} \tilde{X}_0(j)$$ denote the total number of balls in configuration $\tilde{X}$ and $\tilde{X}_0$ respectively. We have $N\leq \tilde{N}\leq\infty$ according to \eqref{hypH2} --with quantities being possibly infinite when  $\delta(X) =-\infty$. We also define the number of balls strictly on the right of bin $j_0$:
$$K \defeq \sum_{j = j_0 + 1}^\infty X_0(j) \quad \hbox{and} \quad \tilde{K} \defeq \sum_{j = j_0 + 1}^\infty \tilde{X}_0(j).$$
We also have $K \leq \tilde{K} < \infty$ thanks to \eqref{hypH2}. Furthermore, if $K = \tilde{K}$ then, we must have $\tilde{X}_0(j) = X_0(j)$ for all $j\in \rrbracket \delta, j_0\rrbracket$. This exactly means that \eqref{hypH1} is satisfied so we are back in the setting of the previous lemma and we can use the trivial coupling to prove the result at the next step. Thus, we now assume that $K < \tilde{K}$.

We also assume without loss of generality that $\mu(\llbracket1, K\rrbracket) < 1$ otherwise the result is again trivial. Because $\mu$ is non-increasing, this assumption is equivalent to requiring $\mu(K+1) > 0$.

Let $\xi$, $\chi$ and $U$ denote three independent random variables such that
\begin{itemize}
\item $\xi$ is distributed according to $\mu$.
\item $\chi$ is distributed according to the conditional law $\mu( \cdot \; | \; \rrbracket K, \tilde{K}\rrbracket)$.
\item $U$ is uniform in $[0,1]$.
\end{itemize}
We construct a random variable $\tilde{\xi}$ depending on $\xi$ in the following way:
\begin{enumerate}
\item[(a)] If $\xi \leq K$, we set $\tilde{\xi} \defeq \xi$.
\item[(b)] If $\xi > \tilde{N}$, we set $\tilde{\xi} \defeq \xi$.
\item[(c)] If $\xi \in \rrbracket N ,  \tilde{N}\rrbracket$, we set $\tilde{\xi} \defeq \chi$.
\item[(d)] If $\xi = k \in \rrbracket K,N\rrbracket$, then $k$ is the rank of a ball in $X_0$. If a ball with the same location and same height also exists in configuration $\tilde{X}_0$, then we denote $\tilde{k}$ its rank in $\tilde{X}_0$. We set
$$
\tilde{\xi} \defeq \begin{cases}
\tilde{k} & \hbox{if $\tilde{k}$ exists and $U < \frac{\mu(\tilde{k})}{\mu(k)}$,}\\
\chi & \hbox{otherwise.}
\end{cases}
$$
\end{enumerate}
We have $K \leq N \leq \tilde{N}$ so the four cases (a), (b), (c) and (d) are disjoint and $\tilde{\xi}$ is well-defined. We now define the coupling:
$$X_1 \defeq \Phi_\xi(X_0) \quad \hbox{and}\quad \tilde{X}_1 \defeq \Phi_{\tilde\xi}(\tilde{X}_0).$$
where $\Phi$ is the operator defined in \eqref{eq:operatorPhi}. We check that \eqref{hypH2} holds for $n=1$. Indeed, we observe that
\begin{itemize}
\item In case (a), configurations $X_1$ and $\tilde{X}_1$ are obtained by adding a ball into a bin, not necessarily the same, but located in $\llbracket j_0 + 2, +\infty\llbracket$ in both cases.
\item In case (b), we have $X_1 = X_0$ and $\tilde{X}_1 = \tilde{X}_0$ if $\delta = -\infty$. Otherwise, the new ball is added in both cases to the bin $\delta +1$.
\item In case (c), we have $X_1 = X_0$ (if $\delta =-\infty$) or $X_1(\delta+1) = X_0(\delta+1)+1$ (if $\delta > -\infty$), while $\tilde{X}_1$ is obtained by adding a ball to $\tilde{X}_0$ into a bin located somewhere in $\llbracket j_0+2,+\infty\llbracket \subset \rrbracket \delta,\infty\llbracket$.
\item In case (d), if $\tilde{k}$ exists and $U < \mu(\tilde{k}) / \mu(k)$  then $X_1$ and $\tilde{X}_1$ are obtained by adding a ball in the same bin for both configuration. Otherwise, a ball is added in a bin located somewhere in $\rrbracket \delta, j_0 + 1\rrbracket$ for configuration $X_0$ whereas a ball is added in a bin located somewhere in $\llbracket j_0 + 2, +\infty\llbracket$ for configuration $\tilde{X}_0$.
\end{itemize}
In all cases, hypothesis \eqref{hypH2} is preserved. It is also clear that \eqref{incH2} holds because a ball is added in $\tilde{X}$ at some location $\leq j_0 + 1$ when $\tilde{\xi} \in \rrbracket\tilde{K}, \tilde{N}\rrbracket$ but this can only occur in case (d) and a ball is always added at the same location in $X$ when this happens.

It remains to prove that $\tilde{\xi}$ has distribution $\mu$.
\begin{itemize}
\item Let $i \leq K$ or $i > \tilde{N}$, we have by construction
$$
\P(\tilde{\xi} = i) = \P(\xi = i) = \mu(i).
$$
\item Let $i \in \rrbracket\tilde{K} , \tilde{N}\rrbracket$. We observe that the ball with rank $i$ in $\tilde{X}_0$ is in a bin at some location $\rrbracket \delta, j_0\rrbracket$. Because of Assumption \eqref{hypH2}, a ball with the same location and same height also exists in $X_0$ and its corresponding rank $k$ in $X_0$ satisfies $k \leq i$. Therefore, we get
$$
\P(\tilde{\xi} = i) = \P\Big(\xi = k,\; U < \frac{\mu(i)}{\mu(k)}\Big) = \mu(k) \P\Big(U < \frac{\mu(i)}{\mu(k)}\Big) = \mu(i),
$$
where we used the fact $0 \leq \mu(i) / \mu(k) \leq 1$ because $(\mu(n))$ is non-increasing.
\item Let $i \in \rrbracket K, \tilde{K}\rrbracket$. We can only have $\tilde{\xi} = i$ in cases (c) or (d) when $\tilde{\xi} = \chi$. But by construction, the event $\{\tilde{\xi} = \chi\}$ depends only on $\xi$ and $U$ which are independent of $\chi$. Therefore,
$$
\P(\tilde{\xi} = i)  =  \P(\tilde{\xi} = \chi,\; \chi = i)
 =  \P(\tilde{\xi} = \chi) \P(\chi = i)
 =  \alpha \frac{\mu(i)}{\mu(\rrbracket K, \tilde{K}\rrbracket)}
$$
for some constant $\alpha \in [0,1]$. Finally, since $\tilde{\xi}$ is a non-defective random variable, we have $\sum \P(\tilde{\xi} = i) = 1$ which implies that $\alpha = \mu(\rrbracket K, \tilde{K}\rrbracket)$ and the proof of the lemma is complete. \qedhere
\end{itemize}
\end{proof}

\subsection{\texorpdfstring{$0-1$}{0-1} law for the freezing of a bin}

We assume here that $(\mu(n), n \geq 1)$ is non-increasing and we show that, for any initial configuration $X_0 \in \S$ and any $k\in \Z$, we have
\begin{equation*}
\P(X_\infty(k) = \infty) \in \{0, 1\}.
\end{equation*}
\begin{proof}[Proof of Theorem~\ref{thm:01law}]
Let $k \in \N$, for all $n \in \N$, we set
\[
B_n \defeq \{ X_{n}(k) > X_{n-1}(k)\},
\]
the event when the $n$th ball added in the process goes into bin $k$. We prove that the events $(B_n, n \geq 1)$ are negatively correlated, which is enough to complete the proof of the $0-1$ law. Indeed, observe that if
\[
  \sum_{n \geq 1} \P(B_n) < \infty,
\]
then by Borel-Cantelli lemma, we have $X_n(k) = X_{n-1}(k)$ a.s. for all $n$ large enough therefore $\P(X_\infty(k) = \infty) = 0$.  Alternatively, if
\[
  \sum_{n \geq 1} \P(B_n) = \infty,
\]
because the events are negatively correlated, we have that $B_n$ holds for infinitely many $n$ a.s.  (c.f. \cite{KochenStone64}), which shows that $\P(X_\infty(k) = \infty) = 1$. Therefore, we just need to prove that, for all $n < m$, we have
\[
  \P(B_n \cap B_m) \leq \P(B_n)\P(B_m).
\]
Note that up to a shift in the initial condition, we can assume without loss of generality that $k=0$. Additionally, using the Markov property, we have
$$\P(B_n \cap B_m|X_1,\ldots, X_{n-1}) = \P_{X_{n-1}}(B_{m-n+1} \cap B_1) = \P_{X_{n-1}}(B_1) \P_{X_{n-1}}(B_{m-n+1}|B_1) $$
with the usual notation $\P_X$ to denote a probability under which the IBM starts from configuration $X$. Therefore, it is enough to show that for any initial configuration, and all $m \geq 1$, we have
\begin{equation}\label{eq:01negcor}
  \P_{X}(B_{m+1}|B_1) \leq \P_{X}(B_{m+1}).
\end{equation}
We write $A, B=B_1, C$ the events where the first ball is added to a bin of negative, null, or positive index respectively. We first check that
\begin{equation}\label{eq:condiBC1}
  \P_{X}(B_{m+1}|B_1) = \P_{X}(B_{m+1}|C).
\end{equation}
This equality follows from Lemma~\ref{lemma:couplage1} with $j_0 \defeq -1$, $X_0$ the configuration obtained from $X$ by adding a ball in a bin at a positive index and
$\tilde{X}_0$ the configuration obtained from $X$ by adding a ball in bin $0$. With this setting, \eqref{eq:condiBC1} is a direct consequence of equivalence \eqref{incH1} applied at $j = j_0 + 1 = 0$.

We also claim that
\begin{equation}\label{eq:condiBC2}
    \P_X(B_{m+1}|A) \geq \P_X(B_{m+1}| B_1).
\end{equation}
This inequality now follows from Lemma~\ref{lemma:couplage2}, with $j_0 \defeq -1$, $X_0$ the configuration obtained from $X$ by adding a ball in a bin of negative index and $\tilde{X}_0$ the configuration obtained from $X$ by adding a ball in bin $0$. With this setting, \eqref{eq:condiBC2} is now a direct consequence of implication \eqref{incH2} applied again at $j = j_0 + 1 = 0$.

Finally, combining \eqref{eq:condiBC1} and \eqref{eq:condiBC2} we find that
\begin{align*}
\P_X(B_{m+1}) & = \P_X(B_{m+1}|A) \P_X(A) + \P_X(B_{m+1}|B) \P_X(B) + \P_X(B_{m+1}|C) \P_X(C) \\
& \geq \P_X(B_{m+1}|B) \big(\P_X(A) +  \P_X(B) + \P_X(C) \big)\\
& = \P_X(B_{m+1}|B)
\end{align*}
which establishes \eqref{eq:01negcor} hence completes the proof of the theorem.
\end{proof}

\subsection{Freezing for finite type distribution and locally bounded configurations}

We prove in this subsection Item 1. of Theorem~\ref{thm:finitevinfinite}. We need the following lemma.

\begin{lemma}\label{lem:pos0} Assume that $(\mu(n), n\geq 1)$ is non-increasing and has finite type $d < \infty$. Let $X_0$ be a locally finite configuration with its front at position $0$. Let $X_0^{(\delta)}$ be  the configuration constructed from $X_0$ by placing an infinite barrier at position $\delta < 0$:
\begin{equation}\label{eq:conflemmaw}
X_0^{(\delta)}(j) \defeq \begin{cases}
X_0(j) & \hbox{if $j\neq \delta$,}\\
+\infty & \hbox{if $j = \delta$.}
\end{cases}
\end{equation}
Then, for any $\delta, k \in \Z$ such that $\delta + d < k \leq 0$, we have
$$
\P( X^{(\delta)}_\infty(k) = X^{(\delta)}_0(k) ) > 0
$$
\emph{i.e.} with positive probability, no ball is ever added to bin $k$ starting from configuration $X_0^{(\delta)}$.
\end{lemma}

\begin{proof}
The result is trivial when $\mu$ has bounded support $\T_\mu \subset \llbracket 0,K \rrbracket$ for some $K\in \N$ because there is probability at least $\mu(1)^K > 0$ that the front of the IBM moves to the right at the first $K$ steps (and then no ball can ever be added at locations $\leq 0$ afterward). We now assume that $\mu$ has unbounded support. Because  $(\mu(n), n\geq 1)$ is non-increasing, this implies that $\mu(n) > 0$ for all $n\in \N$.

Fix $k,\delta$ such that $\delta + d < k \leq 0$. Let $\hat{X}_0^{(\delta)}$ denote the configuration with an infinite bin located at $\delta$ and all other bins empty. Because $X_0^{(\delta)}$ and $\hat{X}_0^{(\delta)}$ differ only by a finite number of balls, we can construct a configuration $Y_0$ such that the IBM($\mu$) can transition from $\hat{X}_0^{(\delta)}$ to $Y_0$ and also from $X_0^{(\delta)}$ to $Y_0$, in both cases with positive probability in a finite number of steps. By definition of $\mu$ being of type $d$ and because $\delta+d < k$, we have $\P(\hat{X}_\infty^{(\delta)}(k) < \infty) = 1$ and therefore, $\P(Y_\infty(k) < \infty) = 1$. By monotone convergence, this means that, for $n_0$ large enough, we have $\P(Y_\infty(k) = Y_{n_0}(k)) > 0$. Observe now that there is only a finite number of configurations that can be reached at time $n_0$ when starting from $Y_0$ (because $Y_0$ has only a finite number of balls on the right of its infinite barrier at $\delta$). Therefore, there exists a configuration  $Z_0$ with $\P(Y_{n_0} = Z_0) > 0$ such that
$$
\P(Z_{\infty}(k) = Z_0(k)) > 0.
$$
By construction, this configuration $Z_0$ can be reached by the IBM($\mu$) starting from $X^{(\delta)}_0$ with positive probability in a finite (say $n$) number of steps. This means that there exists $v_1,v_2,\ldots, v_n \in \N$ such that $\{ X^{(\delta)}_n = Z_0 \} \supset \{\xi_1 = v_1, \ldots, \xi_n = v_n \}$ with the IBM constructed via \eqref{eqn:ibm}. Consider now the configuration $\tilde{Z}_0$ obtained when $\{\xi_1 = \tilde{v}_1, \ldots, \xi_n = \tilde{v}_n \}$ with
$$
\tilde{v}_i \defeq \begin{cases}
    1 & \hbox{if $v_i \leq \sum_{i=k-1}^\infty X^{(\delta)}_{0}(j)$,}\\
    v_i & \hbox{otherwise.}
\end{cases}
$$
Because $\mu(1) > 0$, we have $\P(X^{(\delta)}_n = \tilde{Z}_0) > 0$. By construction, we have $\tilde{Z}_0(k) = X^{(\delta)}_0(k)$ and
$$
 \hbox{for all $j \in \rrbracket\delta , k-1 \rrbracket$.}\qquad \tilde{Z}_0(j) = Z_0(j) \quad\hbox{and}\quad \sum_{i\geq j} \tilde{Z}_0(i) = \sum_{i\geq j} Z_0(i)
$$
Thus, we can invoke Lemma~\ref{lemma:couplage1} with $j_0=k-1$, $Z_0$ and $\tilde{Z}_0$, to deduce from \eqref{incH1} with $j = j_0 + 1 = k$ that
$$
\P(\tilde{Z}_\infty(k) = \tilde{Z}_0(k)) = \P(Z_\infty(k) = Z_0(k)) > 0.
$$
Finally, by the Markov property, we conclude that
$$
\P(X^{(\delta)}_\infty(k) = X^{(\delta)}_{0}(k)) \; \geq \; \P(X^{(\delta)}_{n} = \tilde{Z}_0) \P(\tilde{Z}_\infty(k) = \tilde{Z}_0(k)) \; >\;  0.
$$
\end{proof}

We can now prove Item 1. of Theorem~\ref{thm:finitevinfinite}: we assume again that $(\mu(n), n\geq 0)$ is non-increasing with finite type $d < \infty$ and show that, for any locally finite configuration $X_0$, we have
$$
X_\infty(k) < \infty \quad\hbox{a.s. for all $k\in\Z$.}
$$

\begin{proof}[Proof of Theorem~\ref{thm:finitevinfinite}]
By translation invariance, we can assume without loss of generality that the front of $X_0$ is at location $0$. Furthermore, thanks to Corollary~\ref{cor:cascadefinite}, if bin $k$ freezes, then all bins on its right also freeze so we only need to prove the result for bins at location $k\leq 0$.

Fix $k\leq 0$ and let $X^{(\delta)}_0$ denote again the configuration \eqref{eq:conflemmaw}. We fix $\delta$ such that $\delta + d < k$. We consider the usual coupling between $X^{(\delta)}$ and $X$ where we use the same $\xi_i$'s for both processes. We define $E \defeq \{ X^{(\delta)}_0(k) = X^{(\delta)}_\infty(k) \}$ the event that no ball is ever added to bin $k$ by the IBM starting from $X_0^{(\delta)}$. Lemma~\ref{lem:pos0} ensures that
$$
\P(E) > 0.
$$
We claim that, on $E$, a.s., it holds that, for all $n$,
\begin{equation}\label{eq:induc}
X^{(\delta)}_n(j) = X_n(j) \;\hbox{ for $j\geq k$}\qquad\hbox{and}\qquad \sum_{i = j}^\infty X^{(\delta)}_n(i) \geq \sum_{i = j}^\infty X_n(i)  \;\hbox{ for $\delta < j <  k$.}
\end{equation}
The proof is similar to that of Lemma~\ref{lemma:couplage1} and we proceed by induction. The result is clear for $n=0$. Assume it holds true for some $n$ and consider $\xi_{n+1}$. Because we are on the event $E$, we cannot have $\sum_{i = k}^\infty X^{(\delta)}_n(i) <  \xi_{n+1} \leq  \sum_{i = k - 1}^\infty X^{(\delta)}_n(i)$ as this would add a ball in bin $k$ at time $n+1$ for the process $X^{(\delta)}$. Therefore,
\begin{itemize}
\item Either $\xi_{n+1} \leq \sum_{i = k}^\infty X_n(i) = \sum_{i = k}^\infty X^{(\delta)}_n(i)$ and we add the same ball at the same location $j > k$ for both processes $X$ and $X^{(\delta)}$.
\item Or $\xi_{n+1} > \sum_{i = k - 1}^\infty X^{(\delta)}_n(i)$ in which case $X^{(\delta)}$ adds a ball at some location $j  < k$ and $X$ adds a ball at some location $j' \leq j$ because of the induction hypothesis.
\end{itemize}
Therefore \eqref{eq:induc} holds at time $n+1$ hence at all times.
As a consequence of \eqref{eq:induc}, we find that  $E \subset \{ X_\infty(k) = X_0(k) \}$ a.s. and therefore,
$$
\P( X_\infty(k) < \infty)  \geq \P(X_\infty(k) = X_0(k)) \geq \P(E) > 0.
$$
By the $0-1$ law of Theorem~\ref{thm:01law}, we conclude that $\P( X_\infty(k) < \infty) = 1$ so bin $k$ freezes a.s. and the proof is complete.
\end{proof}

\section{IBM starting with an infinite barrier at \texorpdfstring{$0$}{0}}
\label{sec:typeK}

We work in this section under the assumption that $(\mu(n), n \geq 1)$ is a regularly varying sequence, i.e. that there exists $\alpha \geq 1$ such that for all $u > 0$, we have
\begin{equation}
  \label{eqn:muRegularlyVarying}
  \lim_{n \to \infty} \frac{\mu(\floor{n u})}{\mu(n)}  = u^{-\alpha}.
\end{equation}
We call $-\alpha$ the index of the regularly varying sequence $(\mu(n), n \geq 1)$. Under this condition, we recall that the function
\begin{equation}
  \label{eqn:slowlyVarying}
  L(n) \defeq n^{\alpha} \mu(n)
\end{equation}
is a slowly varying function, i.e. it satisfies $\lim_{n \to \infty} \frac{L(\floor{n u})}{L(n)} = 1$ for all $u > 0$.

We consider $(\hat{X}_n, n \geq 0)$, the IBM($\mu$) starting from the initial condition $\hat{X}_0$ defined by \eqref{def:widehatX} as the configuration with an infinite number of balls in the bin of index $0$ all other bins empty. We show here that under these conditions, for all $k \in \N$, the asymptotic growth rate of $\hat{X}_n(k)$ as $n \to \infty$ is deterministic. This, with a chaining argument, allows us to determine the type of an infinite-bin model. The main result of the section is the following, which is a refinement of Theorem~\ref{thm:typeK}.
\begin{proposition}
\label{prop:laProp}
Assume that $(\mu(n), n \geq 1)$ is regularly varying with index $-\alpha$. Recall the notation $\bar{\mu}(n) \defeq \mu(\llbracket n,+\infty\llbracket)$ for the tail of $\mu$. For all $k \in \N$, we have
$$
\P\Big( \hat{X}_\infty(k) < +\infty\Big) = \begin{cases}
    0 & \hbox{if }\sum_{j = 1}^\infty \bar{\mu}(j)^k = \infty,\\
    1 & \hbox{if }\sum_{j = 1}^\infty \bar{\mu}(j)^k < \infty.
\end{cases}
$$
Moreover,
\begin{enumerate}
\item If $\alpha > 1$, then for all $k\geq 1$ such that $\sum_{j = 1}^\infty \bar{\mu}(j)^k = \infty$, we have
\begin{equation*}
  \lim_{n \to \infty} \frac{\hat{X}_n(k)}{\sum_{j=1}^n \bar{\mu}(j)^k} = \frac{\Gamma\left(1 + \frac{1}{\alpha - 1}\right)}{\Gamma\left(k + \frac{1}{\alpha -1} \right)} \quad \text{a.s.}  
\end{equation*}
\item If $\alpha = 1$, then for all $k \in \N$,
$$\hat{X}_n(k) \underset{n\to\infty}{\sim} nL^{(k)}(n) \quad \hbox{a.s}$$ with $L^{(k)}$ a deterministic slowly varying function.
\end{enumerate}
\end{proposition}

We will make use of the following Borel-Cantelli lemma/law of large numbers for sums of Bernoulli random variables.

\begin{lemma}\label{lem:llnber}
Let $(Z_n, n\geq 1)$ be a sequence of Bernoulli r.v. adapted to some filtration $(\mathcal{F}_n)$. Suppose that there exists a deterministic sequence $(p_n, n\geq 1)$ of positive numbers such that
$$
\lim_{n\to\infty} \frac{\P(Z_{n} = 1\;|\; \mathcal{F}_{n-1})}{p_n} = 1 \quad\hbox{a.s.}
$$
Then, we have
\begin{equation}\label{bernoA}
\P\Big(\sum_{k=n}^\infty Z_k < +\infty\Big) =
\begin{cases}
    1 & \hbox{if }\sum_{n=1}^\infty p_n < \infty,\\
    0 & \hbox{if }\sum_{n=1}^\infty p_n = \infty.
\end{cases}
\end{equation}
Furthermore, in the case that $\sum_{n=1}^\infty p_n = \infty$, we have
\begin{equation}\label{bernoB}
\lim_{n\to\infty}\frac{\sum_{k=1}^n Z_k}{\sum_{k=1}^n p_k} = 1\quad \hbox{a.s.}
\end{equation}
\end{lemma}

This result can be found in Freedman \cite{Freedman}, with the equivalence \eqref{bernoA} being proved in the main theorem, while \eqref{bernoB}, which is referred to as an extension of Lévy's strong law \cite[Chapter~6]{Levy}, is proved in \cite[(40)]{Freedman}. We present here a short self-contained proof, based on an explicit construction that might have independent interest.

\begin{proof}
Equivalence \eqref{bernoA} is also a special case of the Borel-Cantelli Theorem of \cite{Chen78}, using that $Z_n$ is $\mathcal{F}_n$-adapted and
\[
  \sup_{n \geq 1} \frac{Z_n}{1 + Z_1 + \cdots + Z_n} \leq 1 \quad \text{a.s.}
\]
We now assume $\sum_k p_k = +\infty$ and prove \eqref{bernoB}. By considering a possibly enlarged probability space, we can assume that we are also given a sequence $(U_n, n\geq 1)$ of i.i.d. variables with uniform distribution on $[0,1]$ independent of everything else. We define the filtration $\mathcal{G} = (\mathcal{G}_n, n\geq 0)$ by $\mathcal{G}_0 = \sigma(\emptyset)$ the trivial $\sigma$-algebra and $\mathcal{G}_n = \sigma(\mathcal{F}_n , (U_k)_{k \leq n})$. For $n\geq 1$, we set
$$
\tilde{Z}_{n} \defeq \ind{U_{n} \leq \P(Z_{n} = 1\,|\, \mathcal{G}_{n-1})}.
$$
Both processes $Z$ and $\tilde{Z}$ are adapted to the filtration $(\mathcal{G}_n)$ and we have, almost surely,
$$
\P(\tilde{Z}_{n} = 1 \; | \; \mathcal{G}_{n-1}) = \P(U_{n} \leq  \P(Z_{n} = 1\,|\, \mathcal{G}_{n-1}) \, | \, \mathcal{G}_{n-1} ) = \P(Z_{n} = 1 \; | \; \mathcal{G}_{n-1})
$$
Therefore, the processes $Z$ and $\tilde{Z}$ have the same law so we just need to prove \eqref{bernoB} for $\tilde{Z}$. By hypothesis, we have $\P(Z_{n} = 1\;|\; \mathcal{G}_{n-1}) = \P(Z_{n} = 1\;|\; \mathcal{F}_{n-1}) \sim p_n$ almost surely as $n\to\infty$. Thus, if we fix $\varepsilon >0$, we have that $\tilde{Z}_n = \ind{U_{n} \leq \P(Z_{n} = 1\,|\, \mathcal{G}_{n-1})} \leq \ind{U_n \leq (1+\varepsilon)p_n}$ for all $n$ large enough, almost surely. Hence, we deduce that, because $\sum_k p_k$ diverges,
$$
\limsup_{n\to\infty} \frac{\sum_{k=1}^n \tilde{Z}_k}{\sum_{k=1}^n p_k} \leq  \limsup_{n\to\infty} \frac{\sum_{k=1}^n \ind{U_k \leq (1+\varepsilon)p_k}}{\sum_{k=1}^n p_k} = 1+\varepsilon.\quad\hbox{a.s.}
$$
where we used Theorem 2.3.8 of \cite{Durrett2010} to compute the limit on the r.h.s. above since the variables $\ind{U_{k} \leq (1+\varepsilon)p_k}$ are independent Bernoulli variables. The proof for the liminf is identical and we conclude that \eqref{bernoB} holds.
\end{proof}

We now turn to the proof of Proposition~\ref{prop:laProp}, splitting it into two parts. We first treat the case when $\mu$ is regularly varying with index $-\alpha < -1$, before turning to the case $-\alpha = -1$. We recall the statement of the result in that case:
\begin{lemma}
\label{lem:alphaPos}
Assume that $(\mu(n), n \geq 1)$ is regularly varying with index $-\alpha < -1$. For all $k\geq 1$ such that $\sum_{j = 1}^\infty \bar{\mu}(j)^k = \infty$, we have
\begin{equation}
  \label{eqn:growthRatelemma}
  \lim_{n \to \infty} \frac{\hat{X}_n(k)}{\sum_{j=1}^n \bar{\mu}(j)^k} = \frac{\Gamma\left(1 + \frac{1}{\alpha - 1}\right)}{\Gamma\left(k + \frac{1}{\alpha -1} \right)} \quad \text{a.s.}  
\end{equation}
whereas $\hat{X}_\infty(k) < \infty$ a.s. whenever $\sum_{j = 1}^\infty \bar{\mu}(j)^k < \infty$.
\end{lemma}

\begin{proof}
As a preliminary, we observe that by Karamata's theorem (see Feller \cite[Chapter VIII.9, Theorem 1]{Feller71}), the sequence $(\bar{\mu}(j), j \in \N)$ is regularly varying with index $1 - \alpha < 0$, and more precisely
\begin{equation}
  \label{eqn:theSlowlyVarying}
  \bar{\mu}(n) \sim \frac{n^{-(\alpha - 1)}}{\alpha - 1} L(n) \quad \text{ as $n \to \infty$,}
\end{equation}
with $L$ the slowly varying function defined in \eqref{eqn:slowlyVarying}.
More generally, if $k$ is an integer such that $k(\alpha-1) < 1$, using again Karamata's theorem, we have
\begin{equation}
  \label{eqn:theSlowlyVaryingGeneralized}
  \sum_{j=1}^n \bar{\mu}(j)^k \sim \frac{L(n)^k n^{1 + k(1-\alpha)}}{(\alpha-1)^k(1 + k(1-\alpha))}  \quad\text{ as $n \to \infty$}
\end{equation}
while the sum converges for $k(\alpha-1) > 1$. In the corner case $k(\alpha-1) = 1$,  we have
\begin{equation}\label{eqn:theSlowlyVaryingGeneralized2}
\sum_{j=1}^n \bar{\mu}(j)^k \sim \tilde{L}(n)\quad\text{ as $n \to \infty$}
\end{equation}
for some slowly varying function $\tilde{L}$ and the sum may diverge or converge depending $\tilde{L}$.

We now prove Lemma~\ref{lem:alphaPos} by iteratively computing the growth rate of bin $k$ until we reach the first index $k$ satisfying
\[
  \sum_{j=1}^\infty \bar{\mu}(j)^k < \infty.
\]
Recall that $\hat{X}$ is constructed from the sequence $(\xi_n, n\geq 1)$ of i.i.d. variables with law $\mu$ via \eqref{eqn:ibm}. We define the filtration $\mathcal{F}_n = \sigma(\xi_1,\ldots, \xi_n)$. Now, if $\sum_{j=1}^\infty \bar{\mu}(j) = \infty$, applying Lemma~\ref{lem:llnber} (here in the simple case of independent r.v.) we obtain
\begin{equation}
  \label{eqn:growthRateFirst}
  \frac{\hat{X}_n(1)}{\sum_{j=1}^n \bar{\mu}(j)} = \frac{\sum_{j=1}^n \ind{\xi_j \geq j}}{\sum_{j=1}^n \bar{\mu}(j)} \underset{n\to\infty}{\longrightarrow} 1 \quad \text{a.s.},
\end{equation}
proving that $\hat{X}_\infty(1) = \infty$ a.s. and obtaining the growth rate of $\hat{X}_n(1)$ stated in \eqref{eqn:growthRatelemma}. Note that this result holds irrespectively of any regularity assumption on $(\mu(n))$.
Conversely, if $\sum_{j=1}^\infty \bar{\mu}(j) < \infty$, then Lemma~\ref{lem:llnber} states that $\hat{X}_\infty(1) < \infty$ a.s. Thus, we recovered the fact, mentioned in the introduction, that $\mu$ is of type $1$ if and only if it has a first moment.

Next, let $k \geq 2$ such $(k-1)(\alpha-1) < 1$ and assume by induction that \eqref{eqn:growthRatelemma} holds for all $h \in \llbracket 1 , k-1\rrbracket$,  i.e.
\begin{equation}
  \label{eqn:recursionHyp}
  \lim_{n \to \infty} \frac{\hat{X}_n(h)}{\sum_{j=1}^n \bar{\mu}(j)^{h}} = \frac{\Gamma(1 + \frac{1}{\alpha - 1})}{\Gamma(h + \frac{1}{\alpha - 1})} \quad \text{a.s.}
\end{equation}
We deduce from \eqref{eqn:theSlowlyVaryingGeneralized} that, for all $h \in \llbracket 1 , k-1\rrbracket$,
\begin{equation}\label{eqn:recursionHyp2}
  \hat{X}_n(h) \sim \frac{\Gamma(1 + \frac{1}{\alpha - 1})}{\Gamma(h + \frac{1}{\alpha - 1})}  \frac{L(n)^h n^{1 + h(1-\alpha)}}{(\alpha-1)^h(1 + h(1-\alpha))}  \quad \text{a.s. as $n \to \infty$}.
\end{equation}
In particular, $\hat{X}_n(h)$ is negligible compared to $\hat{X}_n(h-1)$ as $n\to\infty$ and we obtain that
\[
  R_n \defeq \sum_{h=1}^{k-1} \hat{X}_n(h) \sim \hat{X}_n(1) \quad \text{a.s. as $n \to \infty$}.
\]
Now, we observe that
\[
  \P(\hat{X}_{n+1}(k) = \hat{X}_n(k) + 1| \mathcal{F}_n) = \mu\Big( \llbracket n - R_n+1, n - R_n + \hat{X}_n(k-1)\rrbracket\Big)
\]
as, at time $n$, there are exactly $n-R_n$ balls to the right of bin $k-1$, thus the balls in that bin are ranked between $n - R_n+1$ and $n - R_n + \hat{X}_n(k-1)$. We have the equivalences
\[
  n - R_n + \hat{X}_{n}(k-1) \; \sim \; n - R_n + 1\; \sim \; n \quad \text{a.s. as $n \to \infty$}.
\]
Because $\mu$ is regularly varying, for all $\epsilon > 0$, there exists $\delta > 0$ such that
\[
  (1 - \delta) \leq \liminf_{n \to \infty} \inf_{m \in [n(1-\epsilon),n(1+\epsilon)]} \frac{\mu(m)}{\mu(n)} \leq \limsup_{n \to \infty} \sup_{m \in [n(1-\epsilon),n(1+\epsilon)]} \frac{\mu(m)}{\mu(n)} \leq (1+\delta).
\]
Hence, we deduce that $\mu\Big( \llbracket n - R_n+1, n - R_n + \hat{X}_n(k-1)\rrbracket\Big) \sim \mu(n)\hat{X}_n(k-1)$ a.s. which means that
\begin{equation*}
  \lim_{n \to \infty} \frac{\P(\hat{X}_{n+1}(k) = \hat{X}_n(k) + 1| \mathcal{F}_n)}{\mu(n)\hat{X}_n(k-1)} = 1 \quad \text{a.s.}
\end{equation*}
We note that, by \eqref{eqn:slowlyVarying}, \eqref{eqn:theSlowlyVarying} and \eqref{eqn:recursionHyp2}, we have, almost surely, as $n \to \infty$,
\begin{align*}
  \mu(n) \hat{X}_n(k-1) &\sim  n^{-\alpha} L(n) \frac{\Gamma(1 + \frac{1}{\alpha - 1})}{\Gamma(k-1 + \frac{1}{\alpha - 1})}  \frac{n^{1 + (k-1)(1-\alpha)}L(n)^{k-1}}{(\alpha-1)^{k-1}(1 + (k-1)(1-\alpha))} \\
  &\sim \frac{\Gamma(1 + \frac{1}{\alpha - 1})}{\Gamma(k + \frac{1}{\alpha - 1})} \bar{\mu}(n)^k.
\end{align*}
As a result,
\begin{equation*}
  \lim_{n \to \infty} \frac{\P(\hat{X}_{n+1}(k) = \hat{X}_n(k) + 1| \mathcal{F}_n)}{\bar{\mu}(n)^k} = \frac{\Gamma(1 + \frac{1}{\alpha - 1})}{\Gamma(k + \frac{1}{\alpha - 1})} \quad \text{a.s.}
\end{equation*}
Invoking Lemma~\ref{lem:llnber} with the sequence $Z_n \defeq \hat{X}_{n+1}(k) - \hat{X}_n(k)$,  we find that, almost surely
$$
\hat{X}_\infty(k) = \sum_{n=0}^{\infty} (\hat{X}_{n+1}(k) - \hat{X}_n(k)) =
\begin{cases}
+ \infty & \hbox{if }\sum_{n=1}^\infty \bar{\mu}(n)^k  = \infty,\\
< \infty & \hbox{if }\sum_{n=1}^\infty \bar{\mu}(n)^k  < \infty.
\end{cases}
$$
Furthermore, when $\sum_{n=1}^\infty \bar{\mu}(n)^k  = \infty$, we obtain the asymptotic
\[
  \lim_{n \to \infty} \frac{\hat{X}_n(k)}{\sum_{j=1}^n \bar{\mu}(j)^k} = \frac{\Gamma(1 + \frac{1}{\alpha - 1})}{\Gamma(k + \frac{1}{\alpha - 1})} \quad \text{a.s.}
\]
hence \eqref{eqn:recursionHyp} holds for $h=k$. This completes the proof of the induction step.

It remains only to treat the special case $k(\alpha -1) = 1$. We have proved above that, in this case, either $\sum_{j=1}^\infty \bar{\mu}(j)^k  < \infty$ and then $\hat{X}_\infty(k) < \infty$ a.s. or $\sum_{j=1}^\infty \bar{\mu}(j)^k  = \infty$ and, in view of \eqref{eqn:theSlowlyVaryingGeneralized2},
$$
\hat{X}_n(k) \sim \sum_{j=1}^n \bar{\mu}(j)^k \sim \tilde{L}(n) \quad \hbox{a.s. as $n\to\infty$}.
$$
which diverges to infinity. In that second case, we repeat the induction step one last time to find now that $\P(\hat{X}_{n+1}(k+1) = \hat{X}_n(k+1) + 1| \mathcal{F}_n) \sim \mu(n)\tilde{L}(n)$. Finally, because $\mu(n)$ is regularly varying with index $-\alpha < -1$, we have $\sum_{j=1}^\infty \mu(j)\tilde{L}(j) < \infty$ from which we conclude, using Lemma~\ref{lem:llnber} one last time, that $\hat{X}_\infty(k+1) < \infty$ a.s. and the proof is now complete.
\end{proof}

Finally, its remains to prove Proposition~\ref{prop:laProp} when $\alpha = -1$. We recall the statement of the result in the lemma below.
\begin{lemma}
\label{lem:alphaNul}
Assume that $(\mu(n), n \geq 1)$ is regularly varying of index $-1$. For all $k \in \N$, we have $\hat{X}_\infty(k) = \infty$ a.s. and furthermore
$$\hat{X}_n(k) \underset{n\to\infty}{\sim} nL^{(k)}(n) \quad \hbox{a.s}$$ with $L^{(k)} : n \mapsto (n \mu(n))^{k-1} \bar{\mu}(n)$ a slowly varying function.
\end{lemma}

\begin{proof}
The proof is based on an identical method as in Lemma~\ref{lem:alphaPos}. We show by induction that for all $k \in \N$, $\hat{X}_n(k)$ is diverging at an almost sure rate which is regularly varying with index~$1$. Note that as $\mu(n)$ is regularly varying of index $-1$, the function $\bar{\mu}(n)$ is now slowly varying.

We first recall that we proved \eqref{eqn:growthRateFirst} without any assumption on the regularity of $\mu$, i.e.
\[
  \hat{X}_n(1) \sim \sum_{j=1}^n \bar{\mu}(j) \quad \text{a.s. as $n \to \infty$},
\]
therefore by Karamata's theorem, we now have $\hat{X}_n(1) \sim n \bar{\mu}(n)$ as expected, with $L^{(1)}(n) = \bar{\mu}(n)$.

Let $k \geq 2$, we assume by induction that $\hat{X}_n(k-1) \sim nL^{(k-1)}(n)$ a.s. as $n \to \infty$. Then with similar computations as in the proof of Lemma~\ref{lem:alphaPos}, we obtain
\[
  \P(\hat{X}_{n+1}(k) = \hat{X}_n(k) + 1| \mathcal{F}_n) \sim \hat{X}_n(k-1) \mu(n) \quad \text{a.s.}
\]
Note that $\hat{X}_n(k-1)\mu(n) \sim (n \mu(n)) L^{(k-1)}(n) = L^{(k)}(n)$ is slowly varying as a product of slowly varying function. Therefore, using Lemma~\ref{lem:llnber} and Karamata's theorem, we deduce that $\hat{X}_n(k) \sim n L^{(k)}(n)$ as $n\to \infty$, which proves the induction hypothesis as the next step. Remark that in particular $\hat{X}_n(k)$ diverges to $\infty$ almost surely.
\end{proof}

\paragraph*{Acknowledgements}

This project would not have been possible without Nicolas Curien, who shared his ideas with us and participated in the early stages of the project. We acknowledge the hospitality of Novosibirsk State University, where part of this work was undertaken in August 2019. This stay was supported by the CNRS-RFBR PRC collaborative grant CNRS-193-382 ``Asymptotic and analytic properties of stochastic ordered graphs and infinite bin models''. SR acknowledges the support and hospitality of the Institut Henri Poincar\'e during the program on “Combinatorics and interactions”, as well as the support of the Fondation Sciences Math\'ematiques de Paris. AS acknowledges the support of ANR 19-CE40-0025 ``ProGraM''. Finally, we wish to thank the referee for their comments that helped improved the previous version of this article.

\nocite{*}
\bibliographystyle{plain}
\bibliography{bibliographie}

\begin{thebibliography}{10}
\bibitem{Freedman}
David Freedman.
\newblock Another note on the {Borel}-{Cantelli} lemma and the strong law, with
  the {Poisson} approximation as a by-product.
\newblock {\em Ann. Probab.}, 1:910--925, 1973.
\end{thebibliography}
\end{document}